\documentclass[12pt]{amsart}
\usepackage{amscd,amsmath,amsthm,amssymb}
\usepackage[left]{lineno}
\usepackage{color}
\usepackage{stmaryrd}
\usepackage[utf8]{inputenc}
\usepackage{cleveref}
\usepackage{epstopdf}
\usepackage{graphicx}
\usepackage{xcolor}
\usepackage{subeqnarray}

\usepackage{cases}
\usepackage{subcaption}

\usepackage{tikz}
\usetikzlibrary{calc}

\definecolor{verylight}{gray}{0.97}
\definecolor{light}{gray}{0.9}
\definecolor{medium}{gray}{0.85}
\definecolor{dark}{gray}{0.6}

\def\NZQ{\mathbb}               

\def\KK{{\NZQ K}}

\def\KK{{\NZQ K}}

\def\G{{\mathcal G}}

\def\0b{{\mathbf 0}}
\def\alphab{{\mathbf \alphab}}
\def\c_ib{{\mathbf c_i}}

\def\reg{{\mathbf reg}}
\def\height{\operatorname{ht}}
\def\depth{\operatorname{depth}}
\def\opn#1#2{\def#1{\operatorname{#2}}} 
\opn\chara{char} \opn\length{\ell} \opn\pd{pd} \opn\rk{rk}
\opn\projdim{proj\,dim} \opn\injdim{inj\,dim} \opn\rank{rank}
\opn\depth{depth} \opn\grade{grade} \opn\height{height}
\opn\embdim{emb\,dim} \opn\codim{codim}

\opn\Tr{Tr} \opn\bigrank{big\,rank}
\opn\superheight{superheight}\opn\lcm{lcm}
\opn\trdeg{tr\,deg}
	\opn\reg{reg} \opn\lreg{lreg} \opn\ini{in} \opn\lpd{lpd}
	\opn\size{size} \opn\sdepth{sdepth}
	\opn\link{link}\opn\fdepth{fdepth}\opn\lex{lex}
	\opn\tr{tr}
	\opn\type{type}
	\opn\gap{gap}
	\opn\arithdeg{arith-deg}
	\opn\HS{HS}
	\opn\GL{GL}
	\opn\div{div} \opn\Div{Div} \opn\cl{cl} \opn\Cl{Cl}
	\opn\Spec{Spec} \opn\Supp{Supp} \opn\supp{supp} \opn\Sing{Sing}
	\opn\Ass{Ass} \opn\Min{Min}\opn\Mon{Mon}
	\opn\Ann{Ann} \opn\Rad{Rad} \opn\Soc{Soc}\opn\Deg{Deg}
	\opn\Im{Im} \opn\Ker{Ker} \opn\Coker{Coker} \opn\Am{Am}
	\opn\Hom{Hom} \opn\Tor{Tor} \opn\Ext{Ext} \opn\End{End}
	\opn\Aut{Aut} \opn\id{id}
	
	\opn\nat{nat}
	\opn\pff{pf}
	\opn\Pf{Pf} \opn\GL{GL} \opn\SL{SL} \opn\mod{mod} \opn\ord{ord}
	\opn\Gin{Gin} \opn\Hilb{Hilb}\opn\sort{sort}
	\opn\PF{PF}\opn\Ap{Ap}
	\opn\mult{mult}
	\opn\bight{bight}
	\opn\Facets{Facets}
	\opn\aff{aff}
	\opn\relint{relint} \opn\st{st}
	\opn\lk{lk} \opn\cn{cn} \opn\core{core} \opn\vol{vol}  \opn\inp{inp} \opn\nilpot{nilpot}
	\opn\link{link} \opn\star{star}\opn\lex{lex}\opn\set{set}
	\opn\width{wd}
	\opn\Fr{F}
	\opn\QF{QF}
	\opn\G{G}
	\opn\type{type}\opn\res{res}
	\opn\conv{conv}
	\opn\Ind{Ind}
	\opn\gr{gr}
	
	\def\pot#1#2{#1[\kern-0.28ex[#2]\kern-0.28ex]}

	\opn\dirlim{\underrightarrow{\lim}}
	\opn\inivlim{\underleftarrow{\lim}}
	\let\to=\rightarrow
	
	\def\Implies{\ifmmode\Longrightarrow \else
		\unskip${}\Longrightarrow{}$\ignorespaces\fi}
	\def\implies{\ifmmode\Rightarrow \else
		\unskip${}\Rightarrow{}$\ignorespaces\fi}
	\def\iff{\ifmmode\Longleftrightarrow \else
		\unskip${}\Longleftrightarrow{}$\ignorespaces\fi}

	\let\:=\colon
	\newtheorem{Theorem}{Theorem}[section]
	\newtheorem{Lemma}[Theorem]{Lemma}

	\newtheorem{Remark}[Theorem]{Remark}

	\newtheorem{Definition}[Theorem]{Definition}

	\let\epsilon\varepsilon
	\let\kappa=\varkappa
	\opn\dis{dis}
	\def\pnt{{\raise0.5mm\hbox{\large\bf.}}}
	
	\opn\Lex{Lex}
	
	\newcommand*{\circled}[1]{\lower.7ex\hbox{\tikz\draw (0pt, 0pt)%
			circle (.5em) node {\makebox[1em][c]{\small #1}};}}

\begin{document}
		\title {$\mathrm{v}$-numbers of powers of edge ideals of cubic circulant graphs}
		
		\author {Guangjun Zhu$^{\ast}$, Jiaxin Li and Yulin Hu}

		\address{School of Mathematical Sciences, Soochow University, Suzhou, Jiangsu, 215006, P. R. China}
		
		\email{zhuguangjun@suda.edu.cn (Corresponding author:Guangjun Zhu),
			\linebreak[4]  lijiaxinworking@163.com(Jiaxin Li), 1246683863@qq.com(Yulin Hu)}

		
		\thanks{$^{\ast}$ Corresponding author}
		


		
		
		\maketitle
		\begin{abstract}
			In this paper, we provide  the exact formula for the $\mathrm{v}$-number of powers of the edge ideal of a cubic circulant graph $C_{2n}(a,n)$ with  $1\le a<n$.
\end{abstract}
		\setcounter{tocdepth}{1}

		\section{Introduction}
		Let $R=\KK[x_1,\dots, x_n]=\bigoplus\limits_{d\ge 0} R_d$ be the standard graded polynomial ring  over a field $\KK$.
		For a graded ideal $I \subseteq R$, let $\operatorname{Ass}(I)$ denote the set of all associated prime ideals of $I$.
		For  each $\frak{p}\in \operatorname{Ass}(I)$,
		the \emph{local $\mathrm{v}$-number} of $I$ at $\frak{p}$
		is defined as
		\[
		\mathrm{v}_{\mathfrak{p}}(I)= \min\{d\ge 0 \mid \text{\ there exists \ } f\in R_d \text{\ such that\ } (I : f)=\mathfrak{p}\}.
		\]
		The \emph{$\mathrm{v}$-number}
		of $I$, denoted by $\mathrm{v}(I)$, is then defined as
		\[
		\mathrm{v}(I)= \min\{\mathrm{v}_{\mathfrak{p}}(I)\mid \mathfrak{p}\in \operatorname{Ass}(I)\}.
		\]
Originally introduced in \cite{CSTVV} to analyze  the asymptotic behavior of the minimum distance of projective Reed-Muller type codes,
the $\mathrm{v}$-number serves as a crucial algebraic invariant connecting coding theory, commutative algebra, and combinatorics.
This invariant has since spurred considerable interest in commutative algebra and algebraic combinatorics, leading to several significant applications. There has been substantial progress in understanding this invariant across different contexts (see \cite{Civan, Conca, FG,GRV,HZL,KNS}). Regarding the asymptotic behavior of $\mathrm{v}$-numbers, Ficarra and Sgroi \cite{FS} and Conca \cite{Conca} independently established that $\mathrm{v}(I^t)$ is eventually linear in $t$ for any graded ideal $I$ in a polynomial ring (and more generally, in an $\mathbb{N}$-graded Noetherian domain \cite{Conca}). We refer the reader to \cite{BMS, F, FS, FG,KNS} for further developments in this direction. In particular, Hien and Vu in \cite{HV} gave a combinatorial description of the stabilization index of the $\mathrm{v}$-function of powers of edge ideals for arbitrary graphs.

	Let $\Gamma$ be a simple graph with vertex set $V(\Gamma)=\{x_1,\ldots,x_n\}$ and edge set $E(\Gamma)$. The \textit{edge ideal} of $\Gamma$, introduced in \cite{V}, is denoted by $I(\Gamma)$ and defined as the ideal of the polynomial ring $R=\mathbb{K}[x_1,\ldots,x_n]$ over a field $\mathbb{K}$ generated by the monomials corresponding to the edges of $\Gamma$. Explicitly,
		\[
		I(\Gamma)=(x_i x_j\mid x_ix_j\in E(\Gamma)).
		\]
The $\mathrm{v}$-number of $I(\Gamma)$ is a fundamental combinatorial invariant of $\Gamma$ that has been used to characterize key topological and structural properties of graphs, such as the family of $W_2$-graphs. Jaramillo and Villarreal in \cite[Theorem 4.5]{JV} showed that for a connected graph $\Gamma$, it is in $W_2$ if and only if $\mathrm{v}(I(\Gamma)) = \dim(R/I(\Gamma))$. Furthermore, determining exact formulas for $\mathrm{v}$-numbers of powers of edge ideals remains an active and challenging task. Chau et al. in    \cite{CHJV}  showed that for a path $P_n$ with $n \ge 2$ vertices and a cycle $C_n$ with $n \ge 3$ vertices,
the $\mathrm{v}$-number of powers of their edge ideals are
\begin{align*}
	\mathrm{v}(I(P_n)^t) &= 2t - 1 + \max \left\{0, \left\lceil \frac{n - 2t - 3}{4} \right\rceil \right\}, \\
\mathrm{v}(I(C_n)^t) &= 2t - 1 + \max \left\{0, \left\lceil \frac{n - 2t - 2}{4} \right\rceil \right\}.
\end{align*}

Circulant graphs have wide applications in coding theory, interconnection networks, and combinatorial designs \cite{CJ, K, M}. Algebraic studies of edge ideals of circulant graphs have been carried out in \cite{AHPZ, BH, HMMP, R, UV, VVW}. Recall that for a positive integer $n \ge 3$ and a subset $S \subseteq \{1, \dots, \lfloor n/2 \rfloor\}$, the circulant graph generated by $S$, denoted by $C_n(S)$, is the graph whose vertex set is $V(C_n(S)) = \{x_1,x_2,\ldots,x_n\}$, in which two distinct vertices $x_i$ and $x_j$ are adjacent if and only if either $|i-j| \in S$ or $n - |i-j| \in S$. In particular, if $S = \{1, \dots, \lfloor \frac{n}{2} \rfloor\}$, then $C_n(S) \cong K_n$, the complete graph on $n$ vertices.
By Fröberg's theorem \cite{Fr}, the edge ideal $I(K_n)$ has a linear resolution, from which it follows from \cite[Theorem 5.1]{F} that $\mathrm{v}(I(K_n)^t)=2t-1$ for all $t\ge 1$. On the other hand, if $S = \{1\}$, then $C_n(1) \cong C_n$, whose $\mathrm{v}$-number formula for powers was completely determined in \cite{CHJV} as mentioned above.

Despite recent progress, explicitly computing the $\mathrm{v}$-number for broader classes of circulant graphs remains highly non-trivial due to their intricate topological structures. A cubic circulant graph is a 3-regular circulant graph of the form $C_{2n}(S)$ with $S = \{a, n\}$ for some integer $a$ such that $1 \le a < n$, denoted by $C_{2n}(a,n)$. Davis and Domke \cite{DD} proved that cubic circulant graphs can be decomposed into disjoint unions of certain copies of circulant graphs of the form $C_{2m}(1,m)$ and $C_{2m}(2,m)$ for odd $m > 2$ (see Lemma \ref{decomposition}).

The main goal of this paper is to bridge this gap by providing precise combinatorial and algebraic formulas for the $\mathrm{v}$-number of cubic circulant graphs. Our main result is as follows:

	\begin{Theorem}\label{main}
		Let $\Gamma=C_{2n}(a,n)$ be a cubic circulant graph with
		$1\leq a<n$, and let $I(\Gamma)$ be its edge ideal. Let
  $k=\gcd(a,n)$, $m=\frac{n}{k}$, $b=\frac{a}{k}$ and
		\[
		\rho=
		\begin{cases}
			m,   & \text{if both $m$ and $b$ are odd,}\\
			m+1, & \text{otherwise.}
		\end{cases}
		\]
		Then, for all integer $t\geq 1$,
		\[
		\begin{aligned}
			\mathrm{v}(I(\Gamma)^t) = \max\Big\{ & 2t+k-2,\; t-1+k\left\lceil\rho/3\right\rceil,\; \left\lceil\frac{4(t-1)+k\rho}{3}\right\rceil, \\
			& \left\lceil\frac{3(t-1)+k\left(\left\lfloor\rho/2\right\rfloor+1\right)}{2}\right\rceil \Big\}.
		\end{aligned}
		\]
		\end{Theorem}

The paper is organized as follows: Section \ref{sec:prelim} introduces the basic facts and preliminary definitions that are used throughout the paper. In section \ref{sec: C{2n}(a,n)} and \ref{sec: C{2n}(1,n)}, we provide exact formulas for the $\mathrm{v}$-numbers of powers of the edge ideals of $\Gamma =C_{2n}(a,n)$ with $1\le a<n$.

\section{Preliminaries}
		\label{sec:prelim}
		
		In this section, we introduce the notation and recall several basic results that will be used throughout the paper.
		Throughout this paper, we adopt the  following symbols: for two non-negative integers  $m,n$,
		let
		\[
		[m,n]=\begin{cases}
			\{m,m+1,\ldots,n\},   &\text{if $m\le n$,}\\
		\emptyset, &\text{if $m> n$. }
		\end{cases}
		\]
 By 	convention, we write $[n]$ for $[1,n]$ and $\prod\limits_{i\in A}x_i=1$ if $A=\emptyset$.

		Let $\Gamma$ be a  finite simple graph with  vertex set $V(\Gamma)$  and edge set $E(\Gamma)$.
		For a vertex $v\in V(\Gamma)$,  its  \emph{neighborhood}  is defined by $N_\Gamma(v)=\{u \in V(\Gamma)\mid uv\in E(\Gamma)\}$. For a subset $A\subseteq V(\Gamma)$, its  \emph{neighborhood}   is defined by $N_\Gamma(A)=\bigcup_{v\in A}N_\Gamma(v)$.
		A subset $C\subseteq V(\Gamma)$ is called a \emph{vertex cover} of $\Gamma$ if every edge of $\Gamma$ has at least one endpoint in $C$. A vertex cover $C$ is called \emph{minimal} if no proper subset of $C$ is a vertex cover of $\Gamma$.
		The \emph{vertex cover number} of $\Gamma$, denoted by $\tau(\Gamma)$, is defined as
			\[
		\tau(\Gamma) = \min\{ |C| : C \text{ is a vertex cover of } \Gamma \}.
		\]
		A subset $S\subseteq V(\Gamma)$ is called an \emph{independent set} of $\Gamma$ if no two distinct vertices in $S$ are adjacent. An independent set $S$ of $\Gamma$ is called \emph{maximal} if every vertex in $V(\Gamma)\setminus S$ is adjacent to at least one vertex in $S$.
		The \emph{independence number} of $\Gamma$, denoted by $\alpha(\Gamma)$, is defined as
		\[
		\alpha(\Gamma) = \max\{ |S| \mid S \text{ is an independent set of } \Gamma \}.
		\]
		It is well known that a set $C\subseteq V(\Gamma)$ is a minimal vertex cover of $\Gamma$ if and only if $V(\Gamma)\setminus C$ is a maximal  independent set of $\Gamma$.
		For any subset $A$ of $V(\Gamma)$, $\Gamma[A]$ denotes the \emph{induced subgraph} of $\Gamma$ on the set $A$, i.e.,
		for any $u,v \in A$, $uv \in E(\Gamma[A])$ if and only if $uv\in E(\Gamma)$.

A \textit{walk} in $\Gamma$ is a sequence of vertices	$W=(x_{i_0},x_{i_1},\ldots,x_{i_k})$
		such that $x_{i_{j-1}}x_{i_j}\in E(\Gamma)$ for each $j\in  [k]$. A \textit{closed walk} is a walk whose initial and terminal vertices coincide, that is, $x_{i_0}=x_{i_k}$. The length of a walk is the number
		of its edges. A walk is called odd (resp. even) if its length is odd (resp. even). A \textit{cycle} of length $k$ in $\Gamma$ is a closed walk
	$C=(x_{i_0},x_{i_1},\ldots,x_{i_k}=x_{i_0})$
	such that $x_{i_0},x_{i_1},\ldots,x_{i_{k-1}}$ are pairwise distinct and $k\ge 3$.

	Let $a,n$ be two positive integers such that $n>a$.  The cubic circulant graph $\Gamma = C_{2n}(a,n)$ has vertex set $V(\Gamma) = \{x_i \mid i\in[2n]\}$, where $x_ix_j$ is an edge if and only if $|i-j| \in \{a, n, 2n-a\}$. Note that $\Gamma$ is connected if and only if  $\gcd(a,n) = 1$. Davis and Domke \cite{DD} proved the following structural result.
		
		\begin{Lemma} \label{decomposition}
			Let $\Gamma = C_{2n}(a,n)$ with $1 \le a < n$, and let $k = \gcd(a,2n)$.
			\begin{itemize}
				\item[(i)]   If $\frac{2n}{k}$ is even, then $\Gamma $ is isomorphic to $k$ copies of $C_{\frac{2n}{k}}(1, \frac{n}{k})$.
				\item[(ii)]  If $\frac{2n}{k}$ is odd, then $\Gamma $ is isomorphic to $\frac{k}{2}$ copies of $C_{\frac{4n}{k}}(2, \frac{2n}{k})$.
			\end{itemize}
		\end{Lemma}

			\begin{Lemma}{\em (\cite[Theorem 4.1]{FM} and \cite[Theorem 4.1(3)]{NT})}
				Let $R_1 = \KK[x_1,\ldots,x_m]$,  $R_2 = \KK[x_{m+1},\ldots,x_n]$ and  $R =\KK[x_1,\ldots,x_m,x_{m+1},\ldots,x_n]$ be three polynomial rings over a field $\KK$,  let
				$I_1 \subseteq R_1$ and $I_2 \subseteq R_2 $ be two monomial ideals, and let $t\ge 1$ be an integer. Then
\[
\operatorname{Ass}_R \bigl((I_1+I_2)^t\bigr) = \bigcup\limits_{i=1}^t \bigcup\limits_{\substack{\mathfrak{p} \in \operatorname{Ass}_{R_1}(I_1^i) \\ \mathfrak{q} \in \operatorname{Ass}_{R_2}(I_2^{t-i+1})}} \{\mathfrak{p} + \mathfrak{q}\}.
\]
Furthermore, if $\mathfrak{p}$ and $\mathfrak{q}$ are prime ideals in $R_1$ and $R_2$, respectively, such that
    $\mathfrak{p} + \mathfrak{q} \in \operatorname{Ass}\bigl((I_1 + I_2)^t\bigr)$, then
    \[
					\mathrm v_{\mathfrak{p}+\mathfrak{q}}\bigl((I_1+I_2)^t\bigr) = \min_{\substack{1 \le i \le t \\ \mathfrak{p} \in \operatorname{Ass}_{R_1}(I_1^{i}) \\ \mathfrak{q} \in \operatorname{Ass}_{R_2}(I_2^{t-i+1})}} (\mathrm v_{\mathfrak{p}}(I_1^{i}) + \mathrm v_{\mathfrak{q}}(I_2^{t-i+1})).
	\]
			\end{Lemma}
			
	By the lemma above, we have 	
		\begin{Lemma}
\label{sum}
		Let $R_1 = \KK[x_1,\ldots,x_m]$ and $R_2 = \KK[x_{m+1},\ldots,x_n]$ be two polynomial rings over a field $\KK$,  let
		$I_1 \subseteq R_1$ and $I_2 \subseteq R_2 $ be two monomial ideals, and let $t\ge 1$ be an integer. Then
			\[
			\mathrm v\bigl((I_1+I_2)^t\bigr)=\min\limits_{1\le j\le t}
				\{\mathrm v(I_1^j)+\mathrm v(I_2^{t+1-j})\}.
		\]
		\end{Lemma}

	 Thus, by Lemmas \ref{decomposition} and  \ref{sum}, determining $\mathrm{v}(I(C_{2n}(a, n))^t)$ for $1 \le a < n$ reduces to considering the cases $a = 1$ and $a = 2$ when $n \ge 3$ is odd.

	\medskip
For a monomial $f = \prod\limits_{i\in[n]}x_i^{s_i} \in R$, we define its degree with respect to $x_{i}$ as $\deg_{x_{i}}(f) = s_{i}$, and its total degree as $\deg(f) = \sum\limits_{i=1}^n s_{i}$.
		For a prime ideal $\mathfrak{p}$ generated by a subset of the variables, we further define
		\[
		\deg_{\mathfrak{p}}(f) = \sum_{x_{i} \in \mathfrak{p}} \deg_{x_{i}}(f).
		\]

		\begin{Lemma}{\em (\cite[Proposition 2.2]{F})}
			\label{lowbound}
			Let $I \subset R$ be a monomial ideal. Then $\mathrm{v}_{\mathfrak{p}}(I) \ge \alpha(I) - 1$  for all  $\mathfrak{p} \in \Ass(I)$,
			where $\alpha(I) = \min \{ \deg(f) \mid f \in I \setminus \{0\} \}$.
		\end{Lemma}

	\begin{Lemma}{\em (\cite[Theorem 3.5]{JV})}
		\label{vedgeideal}
		Let $\Gamma$ be a simple graph and let $I$ be its  edge ideal. Let
		\[
		\mathcal{M}_{\Gamma}=\{A \mid A \text{ is a maximal independent set of } \Gamma\}, \text{ and}
		\]
		\[
		\mathcal{A}_{\Gamma}=\{A \mid A \text{ is an independent set of } \Gamma \text{ and } N_{\Gamma}(A) \text{ is a minimal vertex cover of } \Gamma\}.
		\]
		Then
		\[
		\mathcal{M}_{\Gamma}\subseteq \mathcal{A}_{\Gamma} \text{ \ and \ } \mathrm{v}(I) = \min\{|A| : A \in \mathcal{A}_{\Gamma}\}.
		\]
	\end{Lemma}

	\begin{Lemma}{\em (\cite[Lemma 2.2]{CHJV})}
		\label{colon}
		Let  $t$ be a positive  integer and let $I \subseteq R$ be a nonzero squarefree monomial ideal. Suppose $f$ is a monomial such that $\deg_\mathfrak{p}(f) = t - 1$ and $\mathfrak{p}\subseteq(I^t : f)$
		for some $\mathfrak{p}\in\Ass(I)$, then
		\[
		(I^t : f) = \mathfrak{p}.
		\]
	\end{Lemma}
	
	Let $\Gamma$ be a graph and   $e_1, \dots, e_t$  edges of $\Gamma$. We denote their \emph{$t$-fold product} by $\prod\limits_{i=1}^{t}e_i$.
	To analyze the generators of $(I(\Gamma)^{t+1} : \prod\limits_{i=1}^{t}e_i)$, we introduce the notion of \textit{even-connectedness} with respect to $t$-fold products.
	
		\begin{Definition}{\em (\cite[Definition 6.2]{B})}
		\label{even-connectedness}
		Two vertices $u$ and $v$ {\em ($u$ may be equal to $v$)} are said to be even-connected with respect to a $t$-fold product  $\prod\limits_{i=1}^{t}e_i$ if there is a walk $W:  x_1, x_2, \dots, x_{2k}$   with $k \ge 2$  in $\Gamma$ such that
		\begin{enumerate}
			\item $x_1 = u$ and $x_{2k} = v$.
			\item For each $\ell \in [k-1]$, $x_{2\ell}x_{2\ell+1} = e_i$ for some $i\in[t]$.
			\item For all $i\in[t]$,
			\[
			|\{\ell \in [k-1] \mid x_{2\ell} x_{2\ell+1} = e_i\}| \leq |\{j\in[t] \mid e_j = e_i\}|.
			\]
			\item For all $r \in [2k-1]$, $x_r x_{r+1}$ is an edge in $\Gamma$.
		\end{enumerate}
		If these properties are satisfied, then $x_1, x_2, \dots, x_{2k}$   is said to be an even-connection
		between $u$ and $v$ with respect to $\prod\limits_{i=1}^{t}e_i$.
	\end{Definition}
	
	\begin{Theorem}{\em (\cite[Theorems 6.1 and 6.7]{B})}
		\label{edgecolon}
		Let $t\ge 1$ be an integer,  and let  $\Gamma$ be   a graph and $e_1, \dots, e_t$ be its edges.  Then
		\[
		(I(\Gamma)^{t+1} : \prod\limits_{i=1}^{t}e_i)
		\]
		is generated by monomials of degree two. Moreover, every generator $uv$  {\em($u$ may be equal to $v$)} of $(I(\Gamma)^{t+1} : \prod\limits_{i=1}^{t}e_i)$ is either an edge of $\Gamma$ or even-connected with respect to $\prod\limits_{i=1}^{t}e_i$.
	\end{Theorem}
	
	For a monomial ideal $I$,  let $\mathcal{G}(I)$ be  the unique minimal set of its monomial generators.
	
	\begin{Lemma}\label{t-1edges}
		Let $t\ge 2$ be an integer and let $\Gamma$ be a simple graph.  Let $f\in R$ be a monomial such that $(I(\Gamma)^t:f)=\frak{p}$
		for some $\frak{p}\in \operatorname{Ass}(I(\Gamma))$.
		Then
		\[
		f\in I(\Gamma)^{t-1}.
		\]
	\end{Lemma}
		\begin{proof}
		Let  $x\in \frak{p}$, then $xf \in I(\Gamma)^t$.  Thus we  can write $xf$ as $xf=u_{1} \cdots u_{t}h$, where each $u_i \in \mathcal{G}(I(\Gamma))$ and
		 $h$ is a monomial.
		If $x\mid h$, then $f=u_1\cdots u_t(h/x)\in I(\Gamma)^{t}$. Otherwise, $x\mid u_{\ell}$ for some  $\ell\in[t]$. Thus $f=u_1\cdots u_{\ell-1}(u_{\ell}/x)u_{\ell+1}\cdots u_t h\in I(\Gamma)^{t-1}$.
	\end{proof}

		For a monomial $f=\prod\limits_{i\in[n]}x_i^{s_i}$,  its  \emph{support}  is defined by $\supp(f) = \{x_i\mid i \in [n] \text{ and } s_i\ge 1\}$. For a monomial ideal $I$, we set $\supp(I) = \bigcup\limits_{u\in \mathcal{G}(I)}\supp(u)$.
	\begin{Lemma}\label{NG}
		Let $\Gamma$ be a simple graph, and let $f\in R$ be a monomial such that  $(I(\Gamma):f)=\frak{p}$
		for some $\frak{p}\in \operatorname{Ass}(I(\Gamma))$. Then $\mathfrak{p}=\left(x_i\mid x_i\in N_\Gamma(\supp(f))\right)$.
	\end{Lemma}
	\begin{proof}
		First, we show that $(x_i\mid x_i\in N_{\Gamma}(\supp(f))) \subseteq \mathfrak{p}$. Let $x_i \in N_{\Gamma}(\supp(f))$, then there exists a $x_j \in \supp(f)$ such that $x_ix_j \in E(\Gamma)$.  We can write $f$ as $f = x_j g$ for some monomial $g$. Thus  $x_i f = x_i (x_j g) = (x_i x_j) g\in I(\Gamma)$. Therefore,   $x_i \in (I(\Gamma) : f) = \mathfrak{p}$, it follows that  $(x_i\mid x_i\in N_{\Gamma}(\supp(f))) \subseteq \mathfrak{p}$.
		
	Next, we show that $\mathfrak{p} \subseteq (x_i\mid x_i\in N_{\Gamma}(\supp(f)))$. Let $\mathfrak{p}=(x_{i_1},x_{i_2},\ldots,x_{i_s})$, where $s\ge 1$.   For any $j\in [s]$,  we have $x_{i_j}\in (I(\Gamma):f)$,  which implies that $x_{i_j}f\in I(\Gamma)$. Thus
	$x_u x_v\mid x_{i_j} f$ for some  $x_ux_v \in E(\Gamma)$. Note that  $f\notin I(\Gamma)$,   $x_ux_v\nmid f$. Thus, $x_u=x_{i_j}$ or $x_v=x_{i_j}$.
 Without loss of generality, we can assume $x_u=x_{i_j}$. Thus  $x_v\mid f$,  so $x_v \in \supp(f)$. This implies that  $x_{i_j} \in N_{\Gamma}(x_v)\subseteq N_{\Gamma}(\supp(f))$, since  $x_{i_j}x_v =x_ux_v \in E(\Gamma)$. Therefore, $\mathfrak{p} \subseteq (x_i\mid x_i\in N_{\Gamma}(\supp(f)))$.
	\end{proof}
	
	\begin{Lemma}\label{lowbound1}
		Let   $t\ge 2$ be an integer  and  let 	$\Gamma=C_{2n}(a,n)$ be a cubic circulant graph with  $a\in\{1,2\}$.
 Suppose that for any edges
		$e_1,\ldots,e_{t-1}\in E(\Gamma)$, the colon ideal $(I(\Gamma)^t:\prod\limits_{i=1}^{t-1}e_i)$ can be regarded as the edge ideal of a simple graph.
		Then
		\[
		\mathrm{v}(I(\Gamma)^t)
		\ge 2t-2+
		\left\lceil\frac{\tau(\Gamma)-2t+2}{3}\right\rceil.
		\]
	\end{Lemma}
	\begin{proof}
		Let $I=I(\Gamma)$ and $(I^t:f)=\mathfrak{p}$ for some $\mathfrak{p}\in\Ass(I^t)$  and a
		monomial $f\in R$ such that $\deg(f)=\mathrm{v}(I^t)$.  By Lemma \ref{t-1edges},	$f\in I^{t-1}$.
		Let $f=g\prod\limits_{i=1}^{t-1}e_i$, where $e_1,e_2,\ldots,e_{t-1}\in E(\Gamma)$ and  $g$ is a   monomial.
		Since
		\[
		(I^t:f)=(I^t:g\prod\limits_{i=1}^{t-1}e_i)=((I^t:\prod\limits_{i=1}^{t-1}e_i):g)=(I(H):g)=\mathfrak{p},
		\]
		where $H$ is a simple graph such that $I(H)=(I^t:\prod\limits_{i=1}^{t-1}e_i)$. Then  $g$ is  squarefree and $\mathfrak{p}\in \Ass(I(H))$.
		Therefore,  $\supp(\mathfrak{p})$ is a minimal vertex cover of $H$.  It follows that   $|\supp(\mathfrak{p})|\ge \tau(H)$.
		On the other hand, by Lemma \ref{edgecolon}, each generator $x_ix_j$ of $(I(\Gamma)^{t} : \prod\limits_{i=1}^{t-1}e_i)$ is either an edge of $\Gamma$ or even-connected with respect to $\prod\limits_{i=1}^{t-1}e_i$.
  By Definition
		\ref{even-connectedness} of even-connected,  we know that $V(\Gamma)=V(H)$ and $E(\Gamma)\subseteq E(H)$. Thus  $|\supp(\mathfrak{p})|\ge \tau(H)\ge \tau(\Gamma)$.
		
		\smallskip
	Since $\deg(f)=\deg(g\prod\limits_{i=1}^{t-1}e_i)$ and $\deg(\prod\limits_{i=1}^{t-1}e_i)=2(t-1)$,	to prove that $\deg(f)\ge 2t-2+	\left\lceil\frac{\tau(\Gamma)-2t+2}{3}\right\rceil$, it suffices to show that
		\[
		3\deg(g)+2(t-1)\ge |\supp(\mathfrak{p})|.
		\]
		By Lemma \ref{NG},   $\mathfrak{p}=(x_i\mid x_i\in N_H(\supp(g)))$, so  $|\supp(\mathfrak{p})|=|N_H(\supp(g))|$.
		Note that
		\[
		N_H(\supp(g))=N_{\Gamma}(\supp(g))\sqcup \left(N_H(\supp(g))\setminus N_{\Gamma}(\supp(g))\right),
		\]
		we have
		\[
				|N_H(\supp(g))|=|N_{\Gamma}(\supp(g))|+| N_H(\supp(g))\setminus N_{\Gamma}(\supp(g))|.
		\]
		Therefore, if we prove
		 that
		\begin{align*}
3\deg(g)\ge |N_{\Gamma}(\supp(g))|
\tag{1}\label{eq:myequation1}
			\end{align*}
 and
 \begin{align*}
 2(t-1)\ge |N_H(\supp(g))\setminus N_{\Gamma}(\supp(g))|,
 \tag{2}\label{eq:myequation2}
			\end{align*}
  then the conclusion follows.
		
	Note that $\Gamma$ is $3$-regular,  we have  $|N_{\Gamma}(x_i)|=3$ for any $i\in[2n]$. Since  $g$ is  squareefree,
	it follows that  $3\deg(g)=3|\supp(g)|\ge  |N_{\Gamma}(\supp(g))|$. That is, (\ref{eq:myequation1}) holds.

	Next, we will prove the  relationship (\ref{eq:myequation2}).
		
	Let $d_1,\ldots,d_s$ be the distinct edges appearing in $e_1,\ldots,e_{t-1}$ and let  each $d_i=a_i b_i$. Then $s\le t-1$. It follows from
		\cite[Lemma 3.5]{HPV} that
		\begin{align*}
		I(H)=(I^t:\prod\limits_{i=1}^{t-1}e_i)=(I^{s+1}:\prod\limits_{i=1}^{s}d_i).
		\tag{3}\label{eq:myequation3}
			\end{align*}
	
				For each $u\in N_H(\supp(g))\setminus N_{\Gamma}(\supp(g))$, there exists a $v_u \in \supp(g)$ such that $uv_u \in E(H) \setminus E(\Gamma)$. Thus, by (\ref{eq:myequation3}) and Lemma \ref{edgecolon}, we can deduce that $u$ and $v_u$ are even-connected path with respect to $\prod\limits_{i=1}^{s}d_i$ in $\Gamma$.
				Suppose that this path is $\pi_u:  u = x_{u_1}, x_{u_2}, \dots, x_{u_{2k}} = v_u$ satisfying $k$ as small as possible.
				Since $uv_u\notin E(\Gamma)$ and $k$ is minimus, we have $k\ge2$ and $u\neq x_{u_3}$, which implies $u\in N_{\Gamma}(x_{u_2})\setminus\{x_{u_3}\}$.
				
				By the arbitrariness of $u$, for each $u\in N_H(\supp(g))\setminus N_{\Gamma}(\supp(g))$, we fix an even-connected path $\pi_u:  u = x_{u_1}, x_{u_2}, \dots, x_{u_{2k}} = v_u$ with respect to $\prod\limits_{i=1}^s d_i$ satisfying $k$ as small as possible.  Since $x_{u_2}x_{u_3}=a_{j_u}b_{j_u}$ for some ${j_u}\in [s]$, we have either $x_{u_2}=a_{j_u}$ and  $x_{u_3}=b_{j_u}$, or $x_{u_2}=b_{j_u}$ and $x_{u_3}=a_{j_u}$.  Furthermore,
			$u\in N_{\Gamma}(a_{j_u})\setminus\{b_{j_u}\}$ or 	$u\in N_{\Gamma}(b_{j_u})\setminus\{a_{j_u}\}$.
				
					Next,  for each $j$ in $[s]$,  we define two sets as follows:
				\[
				\begin{aligned}
					A_j
					&=
					\left\{
					u\in N_H(\supp(g))\setminus N_{\Gamma}(\supp(g))
					\;\middle|\;
					\begin{aligned}
						&\text{the even-connection }\pi_u\text{ satisfies}\\[-1mm]
						&x_{u_2}=a_j \text{ and } x_{u_3}=b_j
					\end{aligned}
					\right\},\\[2mm]
					B_j
					&=
					\left\{
					u\in N_H(\supp(g))\setminus N_{\Gamma}(\supp(g))
					\;\middle|\;
					\begin{aligned}
						&\text{the even-connection }\pi_u\text{ satisfies}\\[-1mm]
						&x_{u_2}=b_j \text{ and } x_{u_3}=a_j
					\end{aligned}
					\right\}.
				\end{aligned}
				\]
				As mentioned earlier, $u\in N_{\Gamma}(a_{j_u})\setminus\{b_{j_u}\}$ or	$u\in N_{\Gamma}(b_{j_u})\setminus\{a_{j_u}\}$.	 Thus $A_j\subseteq N_{\Gamma}(a_j)\setminus\{b_j\}$, $B_j\subseteq N_{\Gamma}(b_j)\setminus\{a_j\}$ and $\bigcup\limits_{j\in[s]}(A_j\cup B_j)=N_H(\supp(g))\setminus N_{\Gamma}(\supp(g))$. Furthermore, $|A_j|\le 2$ and $|B_j|\le 2$, since $\Gamma$ is $3$-regular.
				
				Claim: for each $j\in[s]$,  at least one of $A_j$ and $B_j$ is empty.

				Suppose on contradiction that $A_j\neq \emptyset$ and $B_j\neq \emptyset$ for some $j\in[s]$. Then there exist vertices $u_1\in A_j$,  $u_2 \in B_j$ and two  even-connected paths
				\[
				\pi_{u_1}:u_1 \to a_j \to b_j \to \cdots \to v_{u_1},
				\]
				\[
				\pi_{u_2}: u_2 \to b_j \to a_j \to \cdots \to v_{u_2},
				\]
				such that $v_{u_1}, v_{u_2} \in \supp(g)$.  Thus, $b_jv_{u_1}\in E(\Gamma)$, or $b_jv_{u_1}$ is even-connected with respect to $d_1\cdots d_{j-1}\widehat{d_j}d_{j+1}\cdots d_s$, where $\widehat{d_j}$ denotes the element $d_j$ omitted from  $d_1d_2\cdots d_s$. By Lemma \ref{edgecolon}, $v_{u_1}b_j \in (I(\Gamma)^{s}:\prod\limits_{i\in[s]\setminus\{j\}}d_i)$. Similarly, $v_{u_2}a_j \in (I(\Gamma)^{s}:\prod\limits_{i\in[s]\setminus\{j\}}d_i)$.  This implies $v_{u_1} \in (I(\Gamma)^{s}:b_j\prod\limits_{i\in[s]\setminus\{j\}}d_i)$ and  $v_{u_2} \in (I(\Gamma)^{s}: a_j\prod\limits_{i\in[s]\setminus\{j\}}d_i)$.

			\medskip
				If $v_{u_1} = v_{u_2}$, then $v_{u_1} \in (I(\Gamma)^{s}:b_jf_j) \cap (I(\Gamma)^{s}:a_jf_j)$, where  $f_j=\prod\limits_{i\in[s]\setminus\{j\}}d_i$.  Note that
				\begin{align*}
					(I(\Gamma)^{s}:f_ja_j) \cap (I(\Gamma)^{s}:f_jb_j)=I(H).
					\tag{4}\label{eq:myequation4}
				\end{align*}
		by \cite[Lemma 3.7]{HPV}.
			 Thus $v_{u_1} \in I(H)$, a contradiction. Therefore, $v_{u_1}\neq v_{u_2}$.
				By the equality (\ref{eq:myequation4}), we can deduce that $v_{u_1} v_{u_2} \in I(H)$.  Furthermore, $v_{u_1} v_{u_2}\mid g$,
				since $v_{u_1} ,v_{u_2} \in \supp(g)$. Therefore, $g\in I(H)$, which contradicts the fact that $g\notin I(H)$. Therefore, either $A_j=\emptyset$ or $B_j=\emptyset$.
				
				As mentioned earlier, for any $j\in[s]$, we have $|A_j|\le 2$, $|B_j|\le 2$, and  $A_j=\emptyset$ or $B_j=\emptyset$. Therefore,
				\begin{align*}
					|N_H(\supp(g))\setminus N_{\Gamma}(\supp(g))|=|\bigcup\limits_{j\in[s]}(A_j\cup B_j)|\le \sum\limits_{j\in[s]}(|A_j|+|B_j|)
					\le 2s\le 2(t-1).
				\end{align*}
				We complete the proof.
	\end{proof}

\section{The $\mathrm{v}$-number of the edge ideal of $C_{2n}(a,n)$}
\label{sec: C{2n}(a,n)}

In this section, we determine the exact formula for the $\mathrm{v}$-number of the edge ideal of $C_{2n}(a,n)$, where $1\le a<n$. By Lemmas \ref{decomposition} and \ref{sum}, it suffices to consider the case $a=1$ and the case where $a=2$ and $n\ge 3$ is odd. For the sake of convenience, we  stipulate that for any positive integer $j$, if $j\equiv i\  (\text{mod\ }2n)$ with $i\in[2n]$, then  $x_j=x_i$.

 \subsection{The case of $C_{2n}(1,n)$}
We  consider the case $a=1$.  Suppose  the vertex set  and  the  edge set  of  $C_{2n}(1,n)$ are  $V(C_{2n}(1,n))=\{x_i\mid i\in[2n]\}$ and
\[
E(C_{2n}(1,n))=\{x_ix_{i+1}\mid i\in[2n]\}\cup\{x_{j}x_{j+n}\mid j\in[n]\},
\]
respectively.

\begin{Lemma}\label{1vertexcovernumber}
	Let $\Gamma=C_{2n}(1, n)$ be a cubic circulant graph, where $n\ge 2$. Then
	\[
	\tau(\Gamma) =
	\begin{cases}
		n,   &\text{if $n$ is odd,}\\
		n+1, &\text{if $n$ is even.}
	\end{cases}
	\]
\end{Lemma}

\begin{proof}
	Let $C$ be a minimal vertex cover of $\Gamma$. Then $C\cap\{x_{j},x_{j+n}\}\neq \emptyset$ for all $j\in[n]$.
	Since the edges $x_1x_{1+n},x_2x_{2+n},\ldots,x_{n}x_{2n}$ are pairwise disjoint,  we have $|C|\ge n$, so $\tau(\Gamma)\ge n$.
	We consider the following two cases:

(1) 	If $n$ is odd, then $C_1  = \{x_{2k-1}\mid k\in [n]\}$  is a vertex cover of $\Gamma$.  Thus,  $\tau(\Gamma) = n$.
	
(2) 	Suppose that $n$ is even.  We first show that $\tau(\Gamma)\ge n+1$. Suppose to the contrary that $\tau(\Gamma) = n$. Then there exists a  minimal vertex cover $C'$ of $\Gamma$ with $|C'|=n$. Since  $C'\cap\{x_i,x_{i+1}\}\neq\emptyset$ for all $i\in[2n]$,
we can deduce that  either $C' = \{x_{2k-1}\mid k\in [n]\}$ or $C' = \{x_{2k}\mid k\in [n]\}$.
	If $C' = \{x_{2k-1}\mid k\in [n]\}$, then $C' \cap \{x_2, x_{2+n}\} = \emptyset$, which contradicts $x_2x_{2+n}\in E(\Gamma)$.
	Therefore, in this case,  we have  $\tau(\Gamma)\ge n+1$.  By symmetry, if $C' = \{x_{2k}\mid k\in [n]\}$, then we can  also obtain
	$\tau(\Gamma)\ge n+1$.
	
Finally,  let  $C_2=\{x_1\}\cup\left\{x_{2i}\mid i\in \left[\frac{n}{2}\right ]\right\}\cup  \left\{x_{n+2i-1}\mid i\in \left[\frac{n}{2}\right]\right\} $, and  it is easy to check that $|C_2|=n+1$ and $C_2$ is a vertex cover of $\Gamma$. Therefore, $\tau(\Gamma)=n+1$.
\end{proof}

\begin{Theorem}\label{1vnumber}
	Let  $n\ge 2$ be an integer.  Let  $\Gamma=C_{2n}(1, n)$ be a cubic circulant graph, and let $I(\Gamma)$ be  its edge ideal. Then
	\[
	\mathrm{v}(I(\Gamma))=
	\begin{cases}
		\lceil\frac{n}{3}\rceil, &\text{if $n$ is odd,}\vspace{1ex}\\
		\lceil\frac{n+1}{3}\rceil, &\text{if $n$ is even.}
	\end{cases}
	\]
\end{Theorem}
\begin{proof}
	By Lemma \ref{vedgeideal}, there exists an independent set $S$ of $\Gamma$ such that $N_{\Gamma}(S)$ is a minimal vertex cover of $\Gamma$ and $|S|=\mathrm{v}(I(\Gamma))$. Thus  $\tau(\Gamma)\le |N_{\Gamma}(S)|$.    Since  $\Gamma$ is $3$-regular,
	we have $|N_{\Gamma}(S)|\le 3|S|$. Therefore, $\tau(\Gamma)\le  3|S|$. It follows  from Lemma \ref{1vertexcovernumber}
		that
	$\mathrm{v}(I(\Gamma))= |S|\ge \lceil\frac{\tau(\Gamma)}{3}\rceil=\alpha_n$, where  $\alpha_n=
	\begin{cases}
		\lceil\frac{n}{3}\rceil, &\text{if $n$ is odd,}\vspace{1ex}\\
		\lceil\frac{n+1}{3}\rceil, &\text{if $n$ is even.}
	\end{cases}$
	
	In the following, we will prove that  $\mathrm{v}(I(\Gamma))\le \alpha_n$.  Choose $f_2=f_3=x_1$,  $f_4=x_4x_6$, $f_5=	x_1x_{3}$, and
 for any $n\ge 6$,
 	\[
f_n=
\begin{cases}
	x_nx_{n+2}(\prod\limits_{i\in[k-1]} x_{6i+1})(\prod\limits_{j\in[k]} x_{n+6j-2}), & \text{if  $n=6k$}\vspace{1ex},\\
	x_1 x_3 (\prod\limits_{i\in[k]} x_{6i-1}) (\prod\limits_{j\in[k-1]} x_{n+6j+2}), &\text{if $n=6k+1$,}\vspace{1ex}\\
	(\prod\limits_{i\in[k]} x_{6i-3})(\prod\limits_{j\in[k+1]} x_{n+6j-6}), & \text{if  $n=6k+2$,}\vspace{1ex}\\
	\prod\limits_{i\in[2k+1]} x_{6i-5}, &\text{if $n=6k+3$,}\vspace{1ex}\\
	x_n(\prod\limits_{i\in[k]} x_{6i-1}) (\prod\limits_{j\in[k+1]} x_{n+6j-4}), & \text{if  $n=6k+4$,}\vspace{1ex}\\
	x_1 (\prod\limits_{i\in[k+1]} x_{6i-3})(\prod\limits_{j\in[k]} x_{n+6j}), &\text{if $n=6k+5$,}
\end{cases}
\]
where $k\ge 1$.
Then  a simple calculation shows that
	\[
	(I(\Gamma):f_n)=
	\begin{cases}
		(x_2,x_3,x_4), &\text{if $n=2$,}\vspace{1ex}\\
(x_{2i}\mid i\in\left[\frac{n-2}{2}\right])+(x_{n+2j-1}\mid 0\le j\le \frac{n}{2})+(x_{2n}), &\text{if $n\ge 3$ is even,}\vspace{1ex}\\
		(x_{2i}\mid i\in[n]), & \text{if  $n\ge 3$ is odd.}
	\end{cases}
	\]
	By the definition of  the $\mathrm{v}$-number,   we have
	$\mathrm{v}(I(\Gamma))\le \alpha_n$.
\end{proof}

 \subsection{The case of $C_{2n}(2,n)$}
 We will consider the case $a=2$.  Suppose  the vertex set  and  the  edge set  of  $C_{2n}(2,n)$ are  $V(C_{2n}(2,n))=\{x_i\mid i\in[2n]\}$ and
\[
		E(C_{2n}(2,n))=\{x_ix_{i+2}\mid i\in[2n]\}\cup\{x_{j}x_{j+n}\mid j\in[n]\},
\]
respectively.

	By Lemma \ref{decomposition}, it is enough to consider the case where $n\ge 3$ is odd. In this case,
we	can relabel the vertices of $C_{2n}(2,n)$ as follows:
	for each $i\in[n]$,  we replace
	the vertices $x_{2i-1}$ and $x_{n+2i-1}$  with $x_i$ and $y_i$, respectively. Thus,
$V(C_{2n}(2,n))=\{x_i,y_i\mid i\in[n]\}$
and
\begin{align*}
	E(C_{2n}(2,n))=\{x_ix_{i+1},y_iy_{i+1},x_iy_i\mid i\in[n]\}.
	\tag{5}\label{eq:myequation5}
\end{align*}
That is, the graph $C_{2n}(2,n)$ consists of two $n$-cycles formed by the vertex sets $\{x_i \mid i \in [n]\}$ and $\{y_i \mid i \in [n]\}$ with edges $x_i x_{i+1}$ and $y_i y_{i+1}$, respectively. These two cycles are linked by cross-edges of the form $x_i y_i$ for each $i \in [n]$.
To make it easier to understand, we have provided two examples for $n=3$ and $n=5$, see Figure $1$.

\begin{figure}[htbp]
	\centering
	
	\begin{tikzpicture}[scale=1.2,baseline=(u3.center),every node/.style={scale=0.85}]
		\def\R{1.9} 
		\def\r{0.5} 
		
		\coordinate (u1) at (90:\R);
		\coordinate (u2) at (330:\R);
		\coordinate (u3) at (210:\R);
		
		\coordinate (v1) at (90:\r);
		\coordinate (v2) at (330:\r);
		\coordinate (v3) at (210:\r);
		
		\draw (u1) -- (u2) -- (u3) -- cycle;
		
		\draw (v1) -- (v2) -- (v3) -- cycle;
		
		\draw (u1) -- (v1);
		\draw (u2) -- (v2);
		\draw (u3) -- (v3);
		
		\shade [shading=ball, ball color=black] (u1) circle (.06) node [above, yshift=2pt] {$x_1(x_1)$};
		\shade [shading=ball, ball color=black] (u2) circle (.06) node [above right] {$x_3(x_2)$};
		\shade [shading=ball, ball color=black] (u3) circle (.06) node [below left] {$x_5(x_3)$};
		
		\shade [shading=ball, ball color=black] (v1) circle (.06) node [right] {$x_4(y_1)$};
		\shade [shading=ball, ball color=black] (v2) circle (.06) node [below,xshift=-8pt] {$x_6(y_2)$};
		\shade [shading=ball, ball color=black] (v3) circle (.06) node [above left=-3pt] {$x_2(y_3)$};
		
		\node at (0, -\R - 0) {$C_6(2,3)$};
	\end{tikzpicture}
	\hspace{0.3cm} 
	\begin{tikzpicture}[scale=1.05, baseline=(u4.center),every node/.style={scale=0.85}]
		\def\R{1.9}
		\def\r{0.7}
		
		\coordinate (u1) at (90:\R);
		\coordinate (u2) at (18:\R);
		\coordinate (u3) at (306:\R);
		\coordinate (u4) at (234:\R);
		\coordinate (u5) at (162:\R);
		
		\coordinate (v1) at (90:\r);
		\coordinate (v2) at (18:\r);
		\coordinate (v3) at (306:\r);
		\coordinate (v4) at (234:\r);
		\coordinate (v5) at (162:\r);
		
		\draw (u1) -- (u2) -- (u3) -- (u4) -- (u5) -- cycle;
		
		\draw (v1) -- (v2) -- (v3) -- (v4) -- (v5) -- cycle;
		
		\draw (u1) -- (v1);
		\draw (u2) -- (v2);
		\draw (u3) -- (v3);
		\draw (u4) -- (v4);
		\draw (u5) -- (v5);
		
		\shade [shading=ball, ball color=black] (u1) circle (.06) node [above, yshift=2pt] {$x_1(x_1)$};
		\shade [shading=ball, ball color=black] (u2) circle (.06) node [right, xshift=2pt] {$x_3(x_2)$};
		\shade [shading=ball, ball color=black] (u3) circle (.06) node [below right] {$x_5(x_3)$};
		\shade [shading=ball, ball color=black] (u4) circle (.06) node [below left] {$x_7(x_4)$};
		\shade [shading=ball, ball color=black] (u5) circle (.06) node [left, xshift=-2pt] {$x_9(x_5)$};
		
		\shade [shading=ball, ball color=black] (v1) circle (.06) node [above left=-2pt] {$x_6(y_1)$};
		\shade [shading=ball, ball color=black] (v2) circle (.06) node [above,xshift=6pt,yshift=3pt] {$x_8(y_2)$};
		\shade [shading=ball, ball color=black] (v3) circle (.06) node [right] {$x_{10}(y_3)$};
		\shade [shading=ball, ball color=black] (v4) circle (.06) node [below right, xshift=-8pt] {$x_2(y_4)$};
		\shade [shading=ball, ball color=black] (v5) circle (.06) node [below left=-2pt] {$x_4(y_5)$};
		
		\node at (0, -\R - 0.7) {$C_{10}(2,5)$};
	\end{tikzpicture}
	\caption{}
\end{figure}

\begin{Lemma}\label{2oddvertexcovernumber}
	Let $n\ge 3$ be an odd integer, and let $\Gamma=C_{2n}(2,n)$ be a cubic circulant graph. Then $\tau(\Gamma) =n+1$.
\end{Lemma}

\begin{proof}
	Let the vertex set of $\Gamma$ be $V(\Gamma) = \{x_i, y_i \mid i \in [n]\}$, and let the edge set be defined as in (\ref{eq:myequation5}).
	Since $\Gamma$ contains two vertex-disjoint induced odd cycles $C_n^1$ and $C_n^2$ with vertex sets $\{x_i \mid i \in [n]\}$ and $\{y_i \mid i \in [n]\}$, respectively, and each has a vertex cover number of $\frac{n+1}{2}$, we obtain $\tau(\Gamma) \ge \tau(C_n^1) + \tau(C_n^2) = n+1$.
	Conversely, the set
	$$C = \left\{y_{2i-1} \mid i \in \left[\frac{n+1}{2}\right]\right\} \cup \left\{x_{2j} \mid j \in \left[\frac{n-1}{2}\right]\right\} \cup \{x_n\}$$
	forms a vertex cover of $\Gamma$ of size $n+1$.
	Thus, $\tau(\Gamma) = n+1$.
\end{proof}

\begin{Theorem}\label{2oddvnumber}
	Let $n\ge 3$ be an odd integer, and let $\Gamma=C_{2n}(2,n)$ be a cubic circulant graph with $I(\Gamma)$ being its edge ideal. Then
	\[
	\mathrm{v}(I(\Gamma))=\left\lceil\frac{n+1}{3}\right\rceil.
	\]
\end{Theorem}
\begin{proof}
	Let the vertex set of $\Gamma$ be $V(\Gamma) = \{x_i, y_i \mid i \in [n]\}$.
	By Lemma \ref{vedgeideal}, there exists an independent set $S$ of $\Gamma$ such that $N_{\Gamma}(S)$ is a minimal vertex cover of $\Gamma$ and $|S|=\mathrm{v}(I(\Gamma))$.
	Thus, $\tau(\Gamma)\le |N_{\Gamma}(S)|$.
	Since $\Gamma$ is $3$-regular, we have $|N_{\Gamma}(S)|\le 3|S|$.
	Therefore, $\tau(\Gamma)\le 3|S|$.
	It follows from Lemma \ref{2oddvertexcovernumber} that $\mathrm{v}(I(\Gamma))= |S|\ge \left\lceil\frac{\tau(\Gamma)}{3}\right\rceil=\left\lceil\frac{n+1}{3}\right\rceil$.

	Next, we prove that $\mathrm{v}(I(\Gamma))\le \left\lceil\frac{n+1}{3}\right\rceil$.
	Choose $f_3=x_1y_2$,  $f_5=x_1y_4$,  and
	for any odd integer $n\ge 7$,
	\[
	f_n =
	\begin{cases}
		(\prod\limits_{i\in[k]} x_{6i-5}) (\prod\limits_{i\in[k]} y_{6i-2}) y_{6k}, &\text{if $n = 6k + 1$,}\vspace{1ex}\\
		(\prod\limits_{i\in[k+1]} x_{6i-5}) (\prod\limits_{i\in[k]} y_{6i-2}) y_{6k+2}, &\text{if $n = 6k + 3$,}\vspace{1ex}\\
		(\prod\limits_{i\in[k+1]} x_{6i-5}) (\prod\limits_{i\in[k+1]} y_{6i-2}), &\text{if $n = 6k + 5$,}\\
	\end{cases}
	\]
	where $k\ge 1$.
	Then, by straightforward calculation, we obtain that
	\[
	(I(\Gamma):f_n)=\left( y_{2i-1} \middle| i \in \left[\frac{n+1}{2}\right] \right) + \left( x_{2j} \middle| j \in \left[\frac{n-1}{2}\right] \right) + (x_n).
	\]
	By the definition of the $\mathrm{v}$-number, we have
	$\mathrm{v}(I(\Gamma))\le \left\lceil\frac{n+1}{3}\right\rceil$.
\end{proof}

\subsection{The case of $C_{2n}(a,n)$ with $1\le a<n$}

The cases $a=1$ and $a=2$ for odd $n\ge3$ have already been established.  Therefore, by Lemmas \ref{decomposition} and \ref{sum}, and  direct calculation, we can obtain
 the following theorem, which corresponds to the case where $t=1$ in Theorem \ref{main}:

\begin{Theorem}\label{avnumber}
	Let $\Gamma = C_{2n}(a,n)$ be a cubic circulant graph, where $1 \le a < n$, and let $k = \gcd(a, 2n)$. Then
	\[
	\mathrm{v}(I(\Gamma)) =
	\begin{cases}
		k \cdot \left\lceil \frac{n}{3k} \right\rceil, & \text{if $\frac{n}{k}\equiv 1 \pmod{2}$,}\vspace{1ex}\\
		k \cdot \left\lceil \frac{n+k}{3k} \right\rceil, & \text{if $\frac{n}{k}\equiv 0 \pmod{2}$,}\vspace{1ex}\\
		\frac{k}{2} \cdot \left\lceil \frac{2n+k}{3k} \right\rceil, & \text{otherwise.}
	\end{cases}
	\]
\end{Theorem}

	\medskip

\section{The $\mathrm{v}$-number of powers of the edge ideals of $C_{2n}(a,n)$}
\label{sec: C{2n}(1,n)}

In this section, we will provide the exact formula for  the $\mathrm{v}$-numbers of powers of the edge ideal of $\Gamma =C_{2n}(a,n)$ for all $t\ge 2$.
By Lemmas  \ref{decomposition} and \ref{sum}, we only need to consider the case where $a = 1$ and the case where $a = 2$ and $n\ge 3$ is odd.  So, for simplicity, we shall focus solely on these two cases.

 \subsection{The case of $\Gamma=C_{2n}(1,n)$}
	    \label{subsec: C_{2n}(1,n)}
	 	In this subsection, we consider the case where $a = 1$.
When $n=2$, $\Gamma$ is a complete graph with $4$ vertices and so   its edge ideal $I(\Gamma)$ has a linear resolution by \cite[Theorem 5.5]{CZW}. It follows from \cite[Theorem 5.1]{F} that  $\mathrm{v}(I(\Gamma)^t)=2t-1$ for all $t\ge 1$. In the following, we assume that $n\ge 3$.

	    \begin{Theorem}
	\label{main1}
	Let $\Gamma = C_{2n}(1, n)$ be a cubic circulant graph, and let $I(\Gamma)$ be its edge ideal. Then
	\[
	\mathrm{v}(I(\Gamma)^t) = 2t - 1 \quad \text{for all } t \ge \left\lfloor\frac{n}{2}\right\rfloor.
	\]
\end{Theorem}

\begin{proof}
	Let $I = I(\Gamma)$ and $\mu = \lfloor\frac{n}{2}\rfloor$. By Lemma \ref{lowbound}, we have $\mathrm{v}(I^t) \ge 2t - 1$. It remains to prove $\mathrm{v}(I^t) \le 2t - 1$.
To this end,  we choose $f = (x_nx_{n+1})^{t-\mu}\left(\prod_{i=0}^{2\mu-2}x_{n+i}\right)$ and
	\[
	\mathfrak{p} =
	\begin{cases}
		(x_{2i} \mid i \in [\frac{n-2}{2}]) + (x_{n+2j-1} \mid 0 \le j \le \frac{n}{2}) + (x_{2n}), &\text{if $n$ is even,}\vspace{1ex}\\
		(x_{2i} \mid i \in [n]), & \text{if $n$ is odd.}
	\end{cases}
	\]
	Then $\deg(f) = 2t - 1$ and $\mathfrak{p} \in \operatorname{Ass}(I)$ by Theorem \ref{1vnumber}.
	Let $A = \{x_{n+2\ell} \mid 0 \le \ell \le \mu-1\}$, and then it is easy to check that $\mathfrak{p} = (x_i \mid x_i \in N_{\Gamma}(A))$.
	We will prove that $(I^t:f) = \mathfrak{p}$, thus $\mathfrak{p} \in \operatorname{Ass}(I^t)$ and the result holds.
	
	We shall use Lemma \ref{colon} to prove $(I^t:f) = \mathfrak{p}$, so we must verify the conditions of Lemma \ref{colon}.
	Since $\supp(\mathfrak{p}) \cap \supp(f) = \{x_{n+2i-1} \mid i \in [\mu-1]\}$, we have
	\[
	\deg_{\mathfrak{p}}(f) = \sum_{i \in [\mu-1]} \deg_{x_{n+2i-1}}(f) = t - 1.
	\]
	For any $x_i \in \mathfrak{p}$, by the expression of $\mathfrak{p}$, we have $x_i \in N_{\Gamma}(x_{n+2\ell})$, where $0 \le \ell \le \mu-1$.
	Thus
	\[
	x_if = (x_ix_{n+2\ell})\left(\prod_{r=0}^{\ell-1}x_{n+2r}x_{n+2r+1}\right)\left(\prod_{j=\ell+1}^{\mu-1}x_{n+2j-1}x_{n+2j}\right)(x_nx_{n+1})^{t-\mu} \in I^t.
	\]
	Hence $\mathfrak{p} \subseteq (I^t:f)$. Therefore, by  Lemma \ref{colon}, $(I^t : f) = \mathfrak{p}$.
\end{proof}

\begin{Remark}
	\label{admissible1}
From the proof of the above theorem, we observe that
if we choose $U=\emptyset$ and $V=\{x_{n+2\ell}\mid 0\leq \ell\leq \lfloor n/2\rfloor-1\}$, then we can check that $(U,V)$ is an admissible pair as defined in \cite[Definition 1.1]{HV}. Thus,
it follows from \cite[Theorem 1.2]{HV} that $\mathrm{vstab}(I(\Gamma))
\leq \lfloor n/2\rfloor$. Therefore, by \cite[Lemma 2.2]{HV},  for all $t\geq\lfloor n/2\rfloor$, we have $\mathrm{v}(I(\Gamma)^t)=2t-1$.

\textup{However, in  \cite{HV}, Hien and Vu did not provide any techniques for finding admissible pairs. For general graphs, finding admissible pairs and computing $\lambda(U,V)$ in \cite[Theorem 1.2]{HV}  are bothquite challenging; therefore, while their conclusions have significant theoretical value, they are challenging to implement in practice. It was only after completing the proof of the above theorem and studying the associated prime ideal $\mathfrak{p}$ that we discovered the method to choose $U$ and $V$.}
\end{Remark}

\begin{Remark}
	\label{smalln} When $3\le n\le 5$ and $t\ge 1$,
by   Theorems  \ref{1vnumber} and  \ref{main1}, we have
\[
\mathrm{v}(I(\Gamma )^t) =
\begin{cases}
	2,          &\text{if $n\in[4,5]$ and $t=1$,}\\
	2t - 1,   &\text{otherwise.}
\end{cases}
\]
\end{Remark}

In the following, we  will  consider cases where $n\ge 6$ and
 $2 \le t < \lfloor \frac{n}{2} \rfloor$.

\begin{Lemma}\label{simplegraph}
	Let $\Gamma = C_{2n}(1, n)$ be a cubic circulant graph, and  $e_1, \ldots, e_{t-1}$ be its $t-1$ edges.
Suppose $I(\Gamma)$ is  the  edge ideal. of $\Gamma$. Then
 $(I(\Gamma)^{t} : \prod\limits_{i=1}^{t-1}e_i)$ can be regarded as the edge ideal of a simple graph.
\end{Lemma}
\begin{proof}
	By Theorem \ref{edgecolon}, the colon ideal  $(I(\Gamma)^{t} : \prod\limits_{i=1}^{t-1}e_i)$ is generated by monomials of degree two. We only need to show that $x_i^2\notin \mathcal{G}\bigl((I(\Gamma)^{t} : \prod\limits_{i=1}^{t-1}e_i\bigr)$ for any $i\in[2n]$.
	
	Suppose for a contradiction that $x_i^2\in (I(\Gamma)^{t} : \prod\limits_{i=1}^{t-1}e_i)$ for some $i\in[2n]$. Since $\Gamma$  contains no loops,  we have   $x_i^2\notin E(\Gamma)$.   By Theorem \ref{edgecolon}, $x_i$ is even-connected to itself with respect to $\prod\limits_{i=1}^{t-1}e_i$.  Thus, there exists a   walk  $W: x_i = x_{\ell_1}, x_{\ell_2}, \ldots, x_{\ell_{2k}} = x_i$  of  length $2k-1$  satisfying  Definition  \ref{even-connectedness}.  Hence,  $W$ contains an odd cycle.
Furthermore,  by  (2) and (3) of Definition 	\ref{even-connectedness}, we have $k  \le t$.

We consider the following two cases depending on the parity of $n$:
	
\smallskip
	(I) If  $n$ is odd, then  $\Gamma$ is bipartite with bipartition  $V(\Gamma) = A \sqcup B$, where
 $A = \{x_{2i} \mid i\in[n]\}$ and $B = \{x_{2i-1} \mid i\in[n]\}$.   Thus  $\Gamma$ contains no odd cycles, which  contradicts the fact that  $W$ contains an odd cycle.
	
		\smallskip
(II) Suppose $n$ is even. We will prove that  $k\ge \frac{n+2}{2}$.  This will imply  $t\ge k\ge \frac{n+2}{2}$, which  contradicts the fact that
$t < \lfloor \frac{n}{2} \rfloor$.

	To this end, we consider an odd cycle $C$ of length $m$ in $W$ with vertex set $V(C)=\{x_{v_i}\mid i\in[m]\}$ and edge set $E(C)=\{x_{v_i} x_{v_{i+1}}\mid i\in[m]\}$, where $x_{v_{m+1}}=x_{v_1}$.
	For each $i \in [m]$, let $v_{i+1}-v_i \equiv d_i \pmod{2n}$, where $d_i \in \{1, -1, n\}$.
	When we traverse once along $C$, we have $\sum_{i\in[m]}(v_{i+1}-v_i)\equiv 0\pmod{2n}$, that is, $\sum_{i\in[m]}d_i\equiv 0 \pmod{2n}$.

	Let $a=|\{i\in[m] \mid  d_i=1\}|$, $b=|\{i\in[m]\mid  d_i=-1\}|$ and $c=|\{i\in[m]\mid d_i=n\}|$.
	Then $\sum_{i\in[m]}d_i=a-b+cn$ and $m = a + b + c $. Therefore, $a - b + cn =2n\delta$ for some integer $\delta$. Thus $|a-b|=n|2\delta-c|$, which is even since $n$ is even.
	It follows that $a+b$ is even. Since $m$ is odd, $c$ must be odd, which implies $|2\delta-c| \ge 1$ is an odd integer. Consequently,
	\[
	m=(a+b)+c \ge |a-b|+c = |n(2\delta-c)|+c \ge n+1.
\]
Since $W$ contains an odd cycle $C$ of length $m$, we have $2k - 1\ge m \ge n + 1$, yielding $k\ge \frac{n+2}{2}$.
\end{proof}

\begin{Lemma}\label{lowbound2}
	Let $t,n $ be two  integers such that $2 \le t < \lfloor \frac{n}{2} \rfloor$ and $n\ge 6$.
	 Let $\Gamma=C_{2n}(1, n)$ be a cubic circulant graph and   $I(\Gamma)$   its edge ideal. Then
	\[
	\mathrm{v}(I(\Gamma)^t) \ge 2t - 2 +\Theta_{n,t},  \text{\ where \ } \Theta_{n,t}=
	\begin{cases}
		\left\lceil \frac{n - 2t +2}{3} \right\rceil,   &\text{if $n$ is odd,}\\
		\left\lceil \frac{n-2t+3}{3} \right\rceil, &\text{if $n$ is even.}
	\end{cases}
	\]
\end{Lemma}

\begin{proof}
	Let $I=I(\Gamma)$ and  let $e_1, \ldots, e_{t-1}$ be  $t-1$ edges of $\Gamma$.
By Lemma \ref{simplegraph}, the colon ideal $(I^t:\prod\limits_{i=1}^{t-1}e_i)$ can be regarded as  the edge ideal of a simple graph.
The conclusion then follows from Lemmas \ref{lowbound1} and \ref{1vertexcovernumber}.
\end{proof}

\begin{Theorem}
\label{main1(1)}
Let $t,n $ be two  integers such that $2 \le t < \lfloor \frac{n}{2} \rfloor$ and $n\ge 6$.
	 Let $\Gamma=C_{2n}(1, n)$ be a cubic circulant graph and   $I(\Gamma)$   its edge ideal. Then
\[
\mathrm{v}(I(\Gamma)^t) = 2t - 2 +\Theta_{n,t},
\]
 where $\Theta_{n,t}$ is defined as in Lemma \ref{lowbound2}.
\end{Theorem}
\begin{proof}
	Let $I=I(\Gamma)$ and $m=n-2t+2$. By Lemma \ref{lowbound2}, $\mathrm{v}(I^t)\ge 2(t-1)+\Theta_{n,t}$.
	It remains to establish the reverse inequality.
	
	Let
	\[
	f_{t,n}=
	\left(\prod_{i\in[t-1]}x_{2i}x_{2i+1}\right)\tilde f_m,
	\]
	 where $\tilde{f}_4=x_1 x_{n-1}$, $\tilde{f}_5=x_1 x_{2t+1}$, and for $m\ge 6$,
	\[
	\tilde{f}_m =
	\begin{cases}
		x_1 x_{2t+1} (\prod\limits_{i\in[k]} x_{6i+2t-3}) (\prod\limits_{j\in[k-1]} x_{n+2t+6j}), & \text{if $m=6k$ or $m=6k+1$,} \vspace{1ex}\\
		x_1 x_{n-1} (\prod\limits_{i\in[k-1]} x_{6i+2t-1}) (\prod\limits_{j\in[k]} x_{n+2t+6j-4}), & \text{if  $m = 6k+2$,} \vspace{1ex}\\
		x_1 x_{n-2} (\prod\limits_{i\in[k-1]} x_{6i+2t-1}) (\prod\limits_{j\in[k]} x_{n+2t+6j-4}), & \text{if  $m = 6k+3$,} \vspace{1ex}\\
		x_1 x_{n-3} x_{n-1} (\prod\limits_{i\in[k-1]} x_{6i+2t-1}) (\prod\limits_{j\in[k]} x_{n+2t+6j-4}), & \text{if  $m = 6k+4$,}\vspace{1ex}\\
		x_1 (\prod\limits_{i\in[k+1]} x_{6i+2t-5}) (\prod\limits_{j\in[k]} x_{n+2t+6j-2}), & \text{if  $m = 6k+5$,}
	\end{cases}
	\]
	where $k\ge 1$.  Then $	\deg(f_{t,n})=	\deg(\prod\limits_{i\in[t-1]}x_{2i} x_{2i+1})+\deg(\tilde{f}_m) = 2(t-1)+\Theta_{n,t}$.

	Consider the ideal
	\[
	\mathfrak{p}=
	\begin{cases}
		(x_{2i}\mid i\in[n]), & \text{if $n$ is odd},\\
		(x_{2i},x_{n+2i-1}\mid i\in[\frac{n}{2}])+(x_{2n}),
		& \text{if $n$ is even}.
	\end{cases}
	\]
	It is straightforward to verify that $\supp(\mathfrak{p})$ is a vertex cover of $\Gamma$. By Lemma \ref{1vertexcovernumber}, it is a minimal vertex cover. Hence, $\mathfrak{p}\in\Ass(I)$. Moreover, since
	$\supp(\tilde f_m)\cap\supp(\mathfrak{p})=\emptyset$, we have
	\[
	\deg_{\mathfrak{p}}(f_{t,n})
	=\deg_{\mathfrak{p}}\left(
	\prod_{i\in[t-1]}x_{2i}x_{2i+1}
	\right)=t-1.
	\]
	
	By Lemma \ref{colon}, it suffices to prove that
	\[
	\mathfrak{p}\subseteq (I^t:f_{t,n}).
	\]
	
	For $\ell\in[t]$, we have
	\begin{align*}
		x_{2\ell}f_{t,n}&=(\prod\limits_{i\in[\ell]}x_{2i-1}x_{2i})(\prod\limits_{j\in[\ell,t-1]}x_{2j}x_{2j+1})(\frac{\tilde{f}_m}{x_1})\in I^t,\\
		x_{n+2\ell-1}f_{t,n}&=(x_{2\ell-1}x_{n+2\ell-1})(\prod\limits_{i\in[\ell-1]}x_{2i-1} x_{2i})(\prod\limits_{j\in[\ell,t-1]}x_{2j}x_{2j+1})(\frac{\tilde{f}_m}{x_1})\in I^t.
	\end{align*}
	Moreover,   $x_{2n}f_{t,n}=(x_1x_{2n})(\prod\limits_{i\in[t-1]}x_{2i} x_{2i+1})(\frac{\tilde{f}_m}{x_1})\in I^t$. 	Consequently,
	\[
	(x_{2i}, x_{n+2i-1}\mid i\in[t])+(x_{2n})\subseteq (I^t:f_{t,n}).
	\]
	We consider the following two cases depending on the parity of $n$:
	
	\textbf{Case 1.} Suppose that $n$ is odd. Then $m$ is odd and
	\[
	\mathfrak{p}=(x_{2i}, x_{n+2i-1}\mid i\in[t])+(x_{2n})+\mathfrak{p}_1,
	\]	
	 where
	 \[
	 \mathfrak{p}_1= (x_{2i}\mid i\in[t+1,\frac{n-1}{2}])+(x_{n+2j+1}\mid j\in[t,\frac{n-3}{2}]).
	 \]
	 Since
	 \[
	 \frac{f_{t,n}}{\tilde f_m}
	 =\prod_{i\in[t-1]}x_{2i}x_{2i+1}\in I^{t-1},
	 \]
	 it remains to show that
	 \[
	 x\tilde f_m\in I
	 \qquad\text{for each }x\in\mathcal{G}(\mathfrak{p}_1).
	 \]
	
	 	\textbf{Case 2.} Suppose that $n$ is even. Then $m$ is even, and
	 \[
	 \mathfrak{p}
	 =(x_{2i},x_{n+2i-1}\mid i\in[t])
	 +(x_{2n})+\mathfrak{p}_2,
	 \]
	 where
	 \[
	 \mathfrak{p}_2=
	 (x_{2i}\mid i\in[t+1,\tfrac{n}{2}])
	 +(x_{n+2j+1}\mid j\in[t,\tfrac{n-2}{2}]).
	 \]
	 As in Case 1, it suffices to prove that
	 \[
	 x\tilde f_m\in I
	 \qquad\text{for each }x\in\mathcal{G}(\mathfrak{p}_2).
	 \]
	 In both cases, the required containment follows by substituting the corresponding expression for $\tilde f_m$ and factoring $x\tilde f_m$ as an edge monomial times a monomial. More precisely, one checks the possible residues of $m$ modulo $6$: for odd $n$, these are $m=5,6k+1,6k+3,6k+5$; for even $n$, they are $m=4,6k,6k+2,6k+4$.  For the sake of completeness, we shall verify each of these eight cases separately:
	
	 (i) If  $m=5$,  then  $n=2t+3$ and $\mathfrak{p}_1=(x_{2t+2},x_{4t+4})$.  A simple calculation shows that
	 $x_u\tilde{f}_m=(x_u x_{2t+1})(\frac{\tilde{f}_m}{x_{2t+1}})\in I$
	 for $ x_u\in\mathcal{G}(\mathfrak{p}_1)$.
	
	 (ii) If $m=6k+1$, then $n=2t+6k-1$ and $\mathfrak{p}_1=
	 (x_{2i}\mid i\in[t+1,t+3k-1])+(x_{6k+2t+2j}	\mid j\in[t,t+3k-2])$.
	There are three cases:
	 (a) if  $0\le \ell\le k-1$ and $u \in \{2t+2, 2t+4, 4t+6k+2\}$,  then
	 $x_{6\ell+u}\tilde{f}_{m}=(x_{6\ell+u}x_{6\ell+2t+3})(\frac{\tilde{f}_{m}}{x_{6\ell+2t+3}})\in I$;
	 (b) if $j\in [k-1]$ and $v \in \{2t, 4t+6k-2, 4t+6k\}$, then
	 $x_{6j+v}\tilde{f}_{m}=(x_{6j+v}x_{6j+4t+6k-1})(\frac{\tilde{f}_{m}}{x_{6j+4t+6k-1}}) \in I$;
	 (c) $x_{4t+6k}\tilde{f}_{m}= (x_{4t+6k}x_{2t+1})(\frac{\tilde{f}_{m}}{x_{2t+1}}) \in I$.
	
	 (iii) If  $m=6k+3$,  then $n=2t+6k+1$ and $\mathfrak{p}_1=
	 (x_{2i}\mid i\in[t+1,t+3k])+(x_{6k+2t+2j+2}\mid j\in[t,t+3k-1])$.
There are two cases:
	 (a) if $\ell\in [k]$ and $u \in \{2t-2, 2t, 4t+6k\}$, then
	 $x_{6\ell+u}\tilde{f}_{m}=(x_{6\ell+u}x_{6\ell+2t-1})(\frac{\tilde{f}_{m}}{x_{6\ell+2t-1}}) \in I$;
	 (b) if $j\in [k]$ and $v \in \{2t-4, 4t+6k-4, 4t+6k-2\}$, then
	 $x_{6j+v}\tilde{f}_{m}=(x_{6j+v}x_{6j+4t+6k-3})(\frac{\tilde{f}_{m}}{x_{6j+4t+6k-3}}) \in I$.

	 (iv) If  $m=6k+5$, then  $n=2t+6k+3$ and
	 $\mathfrak{p}_1=
	 (x_{2i}\mid i\in[t+1,t+3k+1])+(x_{6k+2t+2j+4}\mid j\in[t,t+3k])$. 		There are three cases:
	 (a) if $u \in \{2t+2, 4t+6k+4\}$, then
	 $x_{u}\tilde{f}_{m}=(x_{u}x_{2t+1})(\frac{\tilde{f}_{m}}{x_{2t+1}}) \in I$;
	 (b) if $2\le \ell\le k+1$ and $v \in \{2t-6, 2t-4, 4t+6k-2\}$, then
	 $x_{6\ell+v}\tilde{f}_{m}=(x_{6\ell+v}x_{6\ell+2t-5})(\frac{\tilde{f}_{m}}{x_{6\ell+2t-5}}) \in I$;
	 (c) if $j\in [k]$ and $w \in \{2t-2, 4t+6k, 4t+6k+2\}$, we have
	 $x_{6j+w}\tilde{f}_{m}=(x_{6j+w}x_{6j+4t+6k+1})(\frac{\tilde{f}_{m}}{x_{6j+4t+6k+1}}) \in I$.
	
	 	(v) If $m=4$, then $n=2t+2$ and $\mathfrak{p}_2=(x_{2t+2},x_{4t+3})$. A simple calculation shows that $x_{u}\tilde{f}_{m}
	 =(x_{u}x_{2t+1})(\frac{\tilde{f}_{m}}{x_{2t+1}})\in I$
	 for $x_u\in\mathcal{G}(\mathfrak{p}_2)$.
	
	 (vi) If $m=6k$, then $n=2t+6k-2$ and $\mathfrak{p}_2=
	 (x_{2i}\mid i\in[t+1,t+3k-1])+
	 (x_{6k+2t+2j-1}\mid j\in[t,t+3k-2])$. There are three cases:
	 (a) if $0\le \ell\le k-1$ and $u \in \{2t+2,2t+4,4t+6k+1\}$, then $x_{6\ell+u}\tilde{f}_{m}
	 =(x_{6\ell+u}x_{6\ell+2t+3})(\frac{\tilde{f}_{m}}{x_{6\ell+2t+3}})\in I$;
	 (b) if $j\in [k-1]$ and $v \in \{2t,4t+6k-3,4t+6k-1\}$, then $x_{6j+v}\tilde{f}_{m}
	 =(x_{6j+v}x_{6j+4t+6k-2})(\frac{\tilde{f}_{m}}{x_{6j+4t+6k-2}})\in I$;
	 (c) $x_{4t+6k-1}\tilde{f}_{m}
	 =(x_{4t+6k-1}x_{2t+1})(\frac{\tilde{f}_{m}}{x_{2t+1}})\in I$.

	 (vii) If $m=6k+2$, then $n=2t+6k$ and $\mathfrak{p}_2=
	 (x_{2i}\mid i\in[t+1,t+3k])+
	 (x_{6k+2t+2j+1}\mid j\in[t,t+3k-1])$.
	There are two cases:
	 (a) if $\ell\in[k]$ and  $u \in \{2t-2,2t,4t+6k-1\}$, then $x_{6\ell+u}\tilde{f}_{m}
	 =(x_{6\ell+u}x_{6\ell+2t-1})(\frac{\tilde{f}_{m}}{x_{6\ell+2t-1}})\in I$;
	 (b) if $j\in[k]$ and
	 $v \in \{2t-4,4t+6k-5,4t+6k-3\}$, then $x_{6j+v}\tilde{f}_{m}
	 =(x_{6j+v}x_{6j+4t+6k-4})(\frac{\tilde{f}_{m}}{x_{6j+4t+6k-4}})\in I$.

	 (viii) If $m=6k+4$, then $n=2t+6k+2$ and $\mathfrak{p}_2=
	 (x_{2i}\mid i\in[t+1,t+3k+1])+(x_{6k+2t+2j+3}\mid j\in[t,t+3k])$.
	There are three cases:
	 (a) if $\ell\in[k]$ and $u \in \{2t-2,2t,4t+6k+1\}$, then $x_{6\ell+u}\tilde{f}_{m}
	 =(x_{6\ell+u}x_{6\ell+2t-1})(\frac{\tilde{f}_{m}}{x_{6\ell+2t-1}})\in I$;
	 (b) if $j\in[k]$ and $v \in \{2t-4,4t+6k-3,4t+6k-1\}$, then $x_{6j+v}\tilde{f}_{m}
	 =(x_{6j+v}x_{6j+4t+6k-2})(\frac{\tilde{f}_{m}}{x_{6j+4t+6k-2}})\in I$;
	 (c) if $w \in \{2t+6k+2,4t+12k+3\}$, then $x_{w}\tilde{f}_{m}=(x_{w}x_{2t+6k+1})(\frac{\tilde{f}_{m}}{x_{2t+6k+1}})\in I$.
	
Consequently,
	 \[
	 \mathfrak{p}_1\subseteq(I:\tilde f_m)
	 \quad\text{or}\quad
	 \mathfrak{p}_2\subseteq(I:\tilde f_m),
	 \]
	 according to the parity of $n$. It follows that
	 \[
	 \mathfrak{p}\subseteq(I^t:f_{t,n}).
	 \]
	 By Lemma \ref{colon}, we conclude that
	 \[
	 (I^t:f_{t,n})=\mathfrak{p}.
	 \]
	 Therefore, $f_{t,n}$ yields
	 \[
	 \mathrm{v}(I^t)
	 \le \deg(f_{t,n})
	 =2(t-1)+\Theta_{n,t}.
	 \]
	 Combining this with the lower bound gives
	 \[
	 \mathrm{v}(I^t)=2t-2+\Theta_{n,t}.
	 \]
	 This completes the proof.
\end{proof}

 \subsection{The case of $\Gamma=C_{2n}(2,n)$}
	    \label{subsec: C_{2n}(2,n)}
		In this subsection, we consider the cases where $a = 2$ and $n\ge 3$ is  odd.

	\begin{Theorem}\label{main2(1)}
	Let $n\ge 3$ be  odd, and  let $\Gamma=C_{2n}(2, n)$ be a cubic circulant graph and   $I(\Gamma)$   its edge ideal. Then
		\[
		\mathrm{v}(I(\Gamma)^t)=2t-1 \quad\text{ for all } t\ge \frac{n+1}{2}.
		\]
	\end{Theorem}
	\begin{proof}
		Let $I=I(\Gamma)$ and $\eta=\frac{n+1}{2}$. As mentioned earlier, for each $i\in[n]$, we relabel the vertices $x_{2i-1}$ and $x_{n+2i-1}$ as $x_i$ and $y_i$, respectively. By Lemma \ref{lowbound}, we have $\mathrm{v}(I^t)\ge 2t-1$. It remains to prove $\mathrm{v}(I^t)\le 2t-1$. 	To this end,  we choose $f=(y_{n-1}y_{n})^{t-\eta}(\prod\limits_{i\in[n]}y_{i})$ and $\mathfrak{m}=(x_i,y_i\mid i\in[n])$. We will prove that $(I^t:f) = \mathfrak{m}$,  the result then follows. Clearly, $(I^t:f)\subseteq \mathfrak{m}$.   For any $i\in[n]$, we have
		\begin{align*}
			x_i f&=(x_i y_i)(y_{n-1}y_{n})^{t-\eta}(\prod\limits_{j\in[\frac{n-1}{2}]}y_{i+2j-1} y_{i+2j})\in I^t,\\
			y_if&=(y_iy_{i+1})(y_{n-1}y_{n})^{t-\eta}(\prod\limits_{j\in[\frac{n-1}{2}]}y_{i+2j} y_{i+2j+1})\in I^t.
		\end{align*}
		Therefore, $\mathfrak{m}\subseteq(I^t:f)$.
	\end{proof}
\begin{Remark}
	\label{admissible2}
From the proof of the above theorem, we observe that
if we choose $U=\emptyset$ and $V=\{y_{2i-1}\mid i\in[\frac{n+1}{2}]\}\cup\{x_n\}$, then we can check that $(U,V)$ is an admissible pair as defined in \cite[Definition 1.1]{HV}. Thus,
it follows from \cite[Theorem 1.2]{HV} that $\mathrm{vstab}(I(\Gamma))
\leq \frac{n+1}{2}$. Therefore, by \cite[Lemma 2.2]{HV},  for all $t\geq\frac{n+1}{2}$, we have $\mathrm{v}(I(\Gamma)^t)=2t-1$.
\end{Remark}

	\begin{Remark}
		\label{smalln2} When $n=3$, by   Theorems  \ref{2oddvnumber} and  \ref{main2(1)}, we have
		\[
		\mathrm{v}(I(C_{2n}(2,n))^t) =
		\begin{cases}
			2,          &\text{if $t=1$,}\\
			2t - 1,   &\text{otherwise.}
		\end{cases}
		\]
	\end{Remark}
	
In the following, we will  consider cases where $n\ge 5$ is odd and $2 \le t < \frac{n+1}{2}$.
	\begin{Lemma}\label{simplegraph1}
		Let $\Gamma=C_{2n}(2, n)$ be a cubic circulant graph, and $e_1, \ldots, e_{t-1}$ be its $t-1$ edges. 	Suppose $I(\Gamma)$ is  the  edge ideal of $\Gamma$. Then $(I(\Gamma)^{t} : \prod\limits_{i=1}^{t-1}e_i)$ can be regarded as the edge ideal of a simple graph.
	\end{Lemma}
	\begin{proof}
		By Theorem \ref{edgecolon}, the colon ideal  $(I(\Gamma)^{t} : \prod\limits_{i=1}^{t-1}e_i)$ is generated by monomials of degree two. We only need to show that $x_i^2\notin \mathcal{G}\bigl((I(\Gamma)^{t} : \prod\limits_{i=1}^{t-1}e_i\bigr)$ for any $i\in[2n]$.
		
		Suppose for a contradiction that $x_i^2\in (I(\Gamma)^{t} : \prod\limits_{i=1}^{t-1}e_i)$ for some $i\in[2n]$. Since $\Gamma$  contains no loops,  we have   $x_i^2\notin E(\Gamma)$. By Theorem \ref{edgecolon}, $x_i$ is even-connected to itself with respect to $\prod\limits_{i=1}^{t-1}e_i$.  Thus, there exists a   walk  $W: x_i = x_{\ell_1}, x_{\ell_2}, \ldots, x_{\ell_{2k}} = x_i$  of  length $2k-1$  satisfying  Definition  \ref{even-connectedness}.  Hence,  $W$ contains an odd cycle. Furthermore,  by  (2) and (3) of Definition 	\ref{even-connectedness}, we have $k  \le t$.
		
		As in the case of Lemma \ref{simplegraph} where $n$ is even, we can prove that $k\ge \frac{n+1}{2}$.  Thus $t\ge k\ge \frac{n+1}{2}$, which  contradicts the fact that 	$t < \frac{n+1}{2}$.
	\end{proof}
	
	The following lemma can be shown by similar arguments as in Lemma \ref{lowbound2}, we omit its proof.
	\begin{Lemma}\label{lowbound3}
	Let $t,n $ be two  integers such that $2 \le t < \frac{n+1}{2}$ and $n\ge 5$ is odd.
	Let $\Gamma=C_{2n}(2, n)$ be a cubic circulant graph and  $I(\Gamma)$   its edge ideal. Then
		\[
		\mathrm{v}(I(\Gamma)^t) \ge 2t - 1 + \left\lceil \frac{n-2t}{3} \right\rceil.
		\]
	\end{Lemma}

	\begin{Theorem}
\label{main2(2)}
		Let $t,n $ be two  integers such that $2 \le t < \frac{n+1}{2}$ and $n\ge 5$ is odd.
		Let $\Gamma=C_{2n}(2, n)$ be a cubic circulant graph and  $I(\Gamma)$   its edge ideal. Then
		\[
		\mathrm{v}(I(\Gamma)^t) = 2t - 1 + \left\lceil \frac{n-2t}{3}\right\rceil.
		\]
	\end{Theorem}
	\begin{proof}

			Let $I=I(\Gamma)$. By Lemma \ref{lowbound3}, $\mathrm{v}(I^t)\ge 2t-1+\left\lceil\frac{n-2t}{3}\right\rceil$.
	It remains to prove the reverse inequality. Let $m=n-2t$. Since $n$ is odd and $t<\frac{n+1}{2}$, the integer $m$ is positive and odd. Let
			\[
			f_{t,n}=x_{2t-1}
			\left(\prod_{i\in[t-1]}x_{2i-1}x_{2i}\right)\tilde f_m,
			\]
				where $\tilde f_1=y_{2t}$, and for $m\ge 3$,
			\[
			\tilde f_m=
			\begin{cases}
				(\prod\limits_{i\in[2,k]}x_{2t+6i-7})
				(\prod\limits_{j\in[k]}y_{2t+6j-4}),& \text{if $m=6k-3$,}\\[2ex]
				(\prod\limits_{i\in[2,k]}x_{2t+6i-7})
				(\prod\limits_{j\in[k]}y_{2t+6j-4})y_{2t+6k-2},& \text{if $m=6k-1$,}\\[2ex]
				(\prod\limits_{i\in[2,k+1]}x_{2t+6i-7})
				(\prod\limits_{j\in[k]}y_{2t+6j-4})y_{2t+6k},& \text{if $m=6k+1$,}
			\end{cases}
			\]
			where $k\ge 1$.  By construction,
			\[
			\deg(f_{t,n})
			=2t-1+\left\lceil\frac{m}{3}\right\rceil
			=2t-1+\left\lceil\frac{n-2t}{3}\right\rceil.
			\]
			
			Consider the ideal
			\[
			\mathfrak p=
			(y_{2i-1}\mid i\in[\tfrac{n+1}{2}])
			+(x_{2j}\mid j\in[\tfrac{n-1}{2}])
			+(x_n).
			\]
			By the proof of Theorem \ref{2oddvnumber}, we have
			$\mathfrak p\in\Ass(I)$. Moreover,
			$\supp(\tilde f_m)\cap\supp(\mathfrak p)=\emptyset$, and hence
			\[
			\deg_{\mathfrak p}(f_{t,n})=t-1.
			\]
			By Lemma \ref{colon}, it suffices to show that
			$\mathfrak p\subseteq(I^t:f_{t,n})$.
				For $\ell\in[t]$, we have
			\begin{align*}
				x_{2\ell}f_{t,n}&=(\prod_{i\in[\ell]}x_{2i-1}x_{2i})
				(\prod_{j\in[\ell,t-1]}x_{2j}x_{2j+1})
				\tilde f_m\in I^{t},\\
				y_{2\ell-1}f_{t,n}&=(x_{2\ell-1}y_{2\ell-1})
				(\prod_{i\in[\ell-1]}x_{2i-1}x_{2i})
				(\prod_{j\in[\ell,t-1]}x_{2j}x_{2j+1})
				\tilde f_m\in I^{t}.
			\end{align*}
			Moreover, we have $x_nf_{t,n}=(x_nx_1)
			(\prod\limits_{i\in[t-1]}x_{2i}x_{2i+1})
			\tilde f_m\in I^{t}$. Therefore,
			\[
			(x_{2i}, y_{2i-1}\mid i\in[t])+(x_{n})\subseteq (I^t:f_{t,n}).
			\]
			Write
		\[
		\mathfrak p=(x_{2i},y_{2i-1}\mid i\in[t])
		+(x_n)+\mathfrak p_1,
		\]
		where
		\[
		\mathfrak p_1=
		(x_{2j}\mid j\in[t+1,\tfrac{n-1}{2}])
		+(y_{2i-1}\mid i\in[t+1,\tfrac{n+1}{2}]).
		\]
		Since
		\[
		x_{2t-1}\prod_{i\in[t-1]}x_{2i-1}x_{2i}
		\in I^{t-1},
		\]
		it remains to verify that
		$u\tilde f_m\in I$ for every
		$u\in\mathcal G(\mathfrak p_1)$.

		We consider the three possible forms of $m$:
		\[
		m=6k-3,\qquad m=6k-1,\qquad m=6k+1,
		\]
		together with the special case $m=1$. For $m=1$, the assertion follows
		immediately from $\tilde f_1=y_{2t}$. For the other three cases,
		substituting the corresponding expression for $\tilde f_m$ shows that
		every generator $u$ of $\mathfrak p_1$ can be paired with a variable
		appearing in $\tilde f_m$ such that their product is an edge monomial
		of $\Gamma$.
		For the sake of completeness, we shall verify each of these three cases separately:
		
		(1) If $m=6k-3$, then $n=2t+6k-3$ and $\mathfrak p_1
		=(x_{2t+2i}\mid i\in[3k-2])+(y_{2t+2j-1}\mid j\in[3k-1])$. There are two cases:
		(i) if $u\in\{x_{6\ell+2t+2},y_{6\ell+2t+1},y_{6\ell+2t+3}\mid 0\le\ell\le k-1\}$, then $u\tilde f_m=(uy_{6\ell+2t+2})
		(\frac{\tilde f_m}{y_{6\ell+2t+2}})\in I$;
		(ii) if $v\in	\{x_{6\ell+2t+4},x_{6\ell+2t+6},y_{6\ell+2t+5}\mid 0\le\ell\le k-2\}$, then $v\tilde f_m=(vx_{6\ell+2t+5})(\frac{\tilde f_m}{x_{6\ell+2t+5}})\in I$.
		
		(2) If $m=6k-1$, then $n=2t+6k-1$ and $\mathfrak p_1
		=(x_{2t+2i}\mid i\in[3k-1])+(y_{2t+2j-1}\mid j\in[3k])$. There are three cases:
		(i) if  $u\in	\{x_{6\ell+2t+2},y_{6\ell+2t+1},y_{6\ell+2t+3}\mid 0\le\ell\le k-1\}$, then $u\tilde f_m=(uy_{6\ell+2t+2})(\frac{\tilde f_m}{y_{6\ell+2t+2}})\in I$;
		(ii) if $v\in	\{x_{6\ell+2t+4},x_{6\ell+2t+6},y_{6\ell+2t+5}\mid 0\le\ell\le k-2\}$, then $v\tilde f_m=(vx_{6\ell+2t+5})(\frac{\tilde f_m}{x_{6\ell+2t+5}})\in I$;
		(iii) if $w\in\{x_{n-1},y_{n-2},y_n\}$, then $w\tilde f_m=(wy_{n-1})
		(\frac{\tilde f_m}{y_{n-1}})\in I$.
		
		(3) If $m=6k+1$, then $n=2t+6k+1$ and $\mathfrak p_1
		=(x_{2t+2i}\mid i\in[3k])+(y_{2t+2j-1}\mid j\in[3k+1])$. There are three cases:
		(i) if $u\in\{x_{6\ell+2t+2},y_{6\ell+2t+1},y_{6\ell+2t+3}\mid 0\le\ell\le k-1\}$, then $u\tilde f_m=(uy_{6\ell+2t+2})
		(\frac{\tilde f_m}{y_{6\ell+2t+2}})\in I$;
		(ii) if $v\in\{x_{6\ell+2t+4},x_{6\ell+2t+6},y_{6\ell+2t+5}\mid 0\le\ell\le k-1\}$, then $v\tilde f_m=(vx_{6\ell+2t+5})
		(\frac{\tilde f_m}{x_{6\ell+2t+5}})\in I$.	(iii) if $w\in\{x_{n-1},y_{n-2},y_n\}$, then $w\tilde f_m=(wy_{n-1})
		\frac{\tilde f_m}{y_{n-1}}\in I$.
		
		Consequently, $u\tilde f_m\in I$ for every
		$u\in\mathcal G(\mathfrak p_1)$, and therefore
		\[
		\mathfrak p_1\subseteq(I:\tilde f_m).
		\]
	It follows that $\mathfrak p\subseteq(I^t:f_{t,n})$.
			By Lemma \ref{colon},
			\[
			(I^t:f_{t,n})=\mathfrak p.
			\]
			Hence
			\[
			\mathrm{v}(I^t)
			\le\deg(f_{t,n})
			=2t-1+\left\lceil\frac{n-2t}{3}\right\rceil.
			\]
			Combining this with the lower bound yields
			\[
			\mathrm{v}(I^t)
			=2t-1+\left\lceil\frac{n-2t}{3}\right\rceil.
			\]
			This completes the proof.
	\end{proof}

	\subsection{The case of $\Gamma=C_{2n}(a,n)$ with $1\le a<n$}

The cases $a=1$ and $a=2$ for odd $n\ge3$ have already been established.  Therefore, by  Lemma \ref{sum} and Theorems  \ref{main1},
\ref{main1(1)}, \ref{main2(1)}, and \ref{main2(2)}, and  direct calculation, we can obtain
 the following theorem, which corresponds to the case where $t\ge 2$ in Theorem \ref{main}:

	\begin{Theorem}\label{mainzhu}
		Let $\Gamma=C_{2n}(a,n)$ be a cubic circulant graph with
		$1\leq a<n$, and let $I(\Gamma)$ be its edge ideal. Let
  $k=\gcd(a,n)$, $m=\frac{n}{k}$, $b=\frac{a}{k}$ and
		\[
		\rho=
		\begin{cases}
			m,   & \text{if both $m$ and $b$ are odd,}\\
			m+1, & \text{otherwise.}
		\end{cases}
		\]
		Then, for all integer $t\geq 1$,
		\[
		\begin{aligned}
			\mathrm{v}(I(\Gamma)^t) = \max\Big\{ & 2t+k-2,\; t-1+k\left\lceil\rho/3\right\rceil,\; \left\lceil\frac{4(t-1)+k\rho}{3}\right\rceil, \\
			& \left\lceil\frac{3(t-1)+k\left(\left\lfloor\rho/2\right\rfloor+1\right)}{2}\right\rceil \Big\}.
		\end{aligned}
		\]
		\end{Theorem}

		\medskip
			\vspace{1mm}

		\hspace{-6mm} {\bf Acknowledgment: }
		This research is supported by the Natural Science Foundation of Jiangsu Province (No. BK20221353) and the National Natural Science Foundation of China (No.12471246).  The first author carried out part of this work while visiting the International Center for Research in Mathematics and Postgraduate Training (ICRTM), affiliated with the Institute of Mathematics at the Vietnam Academy of Science and Technology, through the ICRTM--TWAS Visiting Program. She expresses her sincere gratitude for the financial support and warm hospitality received during this period.

	\end{document}